\documentclass[a4,11pt]{article}
\usepackage{amsthm}

\usepackage[a4paper]{geometry}
\newcommand{\newnumbered}[2]{\newtheorem{#1}[theorem]{#2}}
\newcommand{\newunnumbered}[2]{\newtheorem{#1}[theorem]{#2}}
\newcommand{\Title}[2]{\title{#1}\newcommand{\Acknowledgements}{\section*{Acknowledgements} #2}}
\newcommand{\Author}[2][]{\author{#2}}
\newcommand{\Comma}{\and}
\newcommand{\Und}{\and}
\newcommand{\br}{, }
\newcommand{\fs}{. }
\newcommand{\thanksone}[3][]{#1\thanks{#3\email{\tt #2}}}
\newcommand{\thankstwo}[3][]{#1\thanks{#3\email{\tt #2}}}
\newcommand{\thanksthree}[4][]{#1\thanks{#3} \thanks{#4\email{\tt #2}}}
\newcommand{\email}[1]{#1}
\newcommand{\classno}[2][2000]{}
\newcommand{\printscl}{}
\newcommand{\mktitle}{\maketitle}
\newcommand{\mkabstitle}{}

\newcommand{\bqed}{}

\usepackage[utf8x]{inputenc}
\usepackage{amsfonts}
\usepackage{amsmath}
\usepackage{amssymb}
\usepackage[mathscr]{eucal}
\usepackage{color}
\usepackage[colorlinks=true,citecolor=black,linkcolor=black,urlcolor=blue]{hyperref}
\usepackage{comment}
\newtheorem{theorem}{Theorem}[section] 
\newtheorem{lemma}[theorem]{Lemma}     
\newtheorem{corollary}[theorem]{Corollary}

\newnumbered{assertion}{Assertion}    
\newnumbered{conjecture}{Conjecture}  
\newnumbered{definition}{Definition}
\newnumbered{hypothesis}{Hypothesis}
\newnumbered{remark}{Remark}
\newnumbered{note}{Note}
\newnumbered{observation}{Observation}
\newnumbered{problem}{Problem}
\newnumbered{question}{Question}
\newnumbered{algorithm}{Algorithm}
\newnumbered{example}{Example}
\newunnumbered{notation}{Notation} 
\numberwithin{equation}{section}

\newcommand{\DOI}[1]{\href{http://dx.doi.org/#1}{\texttt{doi:#1}}}

\newcommand{\noop}[1]{}
\newcommand{\mm}{\ensuremath{\!-\!}}
\newcommand{\pp}{\ensuremath{\!+\!}}

\newcommand{\OA}{\ensuremath{\operatorname{OA}}}
\newcommand{\sV}[2]{{
	\setlength{\arraycolsep}{2pt}
	\renewcommand{\arraystretch}{0.8}
	\left[\begin{array}{ccc} #1 \\ #2 \end{array}\right]
}}

\begin{document}

\Title{On tight $4$-designs in Hamming association schemes}{%
Alexander Gavrilyuk is supported by
BK21plus Center for Math Research and Education at Pusan National University,
and by Basic Science Research Program through the National Research Foundation of Korea (NRF) funded
by the Ministry of Education (grant number NRF-2018R1D1A1B07047427).
Sho Suda is supported by JSPS KAKENHI Grant Number 18K03395.
Jano\v{s} Vidali is supported by the Slovenian Research Agency
(research program P1-0285 and project J1-8130).
}

\Author[Alexander Gavrilyuk, Sho Suda and Jano\v{s} Vidali]{%
\thanksone[Alexander Gavrilyuk]{alexander.gavriliouk@gmail.com}{%
Center for Math Research and Education\br
Pusan National University\br
2, Busandaehak-ro 63beon-gil\br
Geumjeong-gu, Busan, 46241\br
Republic of Korea\fs
}%
\Comma
\thankstwo[Sho Suda]{suda@auecc.aichi-edu.ac.jp}{%
Department of Mathematics Education\br
Aichi University of Education\br
1 Hirosawa, Igaya-cho, Kariya\br
Aichi, 448-8542\br
Japan\fs
}
\Und
\thanksthree[Jano\v{s} Vidali]{janos.vidali@fmf.uni-lj.si}{%
Faculty of Mathematics and Physics\br
University of Ljubljana\br
Jadranska ulica 21\br
1000 Ljubljana\br
Slovenia\fs
}{%
Institute of Mathematics, Physics and Mechanics\br
Jadranska ulica 19\br
1000 Ljubljana\br
Slovenia\fs
}
}

\classno{05E30 (primary), 05B15 (secondary)}

\date{\today}

\mktitle

\begin{abstract}
We complete the classification of tight $4$-designs in Hamming association schemes $H(n,q)$, i.e.,
that of tight orthogonal arrays of strength $4$, which had been open since a result by Noda (1979).
To do so, we construct an association scheme attached to a tight $4$-design in $H(n,q)$
and analyze its triple intersection numbers to conclude the non-existence in all open cases.
\printscl
\end{abstract}

\mkabstitle

\section{Introduction}
An {\it orthogonal array} with parameters $(N,n,q,t)$
($\OA(N,n,q,t)$ for short)
is an $N\times n$ matrix with entries from the alphabet $\{1,2,\ldots,q\}$
such that in any its $t$ columns,
all possible row vectors of length $t$ occur equally often.
Since orthogonal arrays were introduced by Rao~\cite{R} in 1947,
they became one of the central topics in combinatorics
and found many applications in related areas
such as coding theory, cryptography, etc., see~\cite{HSS}.

In the theory of orthogonal arrays, a fundamental problem is constructing orthogonal arrays
with extremal parameters. In particular, given the strength $t$, the alphabet size $q$, and the number of columns $n$,
we are interested in orthogonal arrays with minimum possible number of rows $N$.
For $t=2e$, the lower bound on $N$ was given by Rao~\cite{R} as
\begin{align}\label{eq:tight}
N\geq \sum_{k=0}^e \binom{n}{k}(q-1)^k.
\end{align}
An orthogonal array is said to be {\it complete} or {\it tight}
if it achieves equality in this bound.

The rows of an orthogonal array $\OA(N,n,q,t)$
can be naturally considered as a subset of points
of the Hamming association scheme $H(n,q)$,
which form a {\em $t$-design}, a design of strength $t$
(we refer the reader to Section~\ref{sect:Definitions}
for the precise definitions),
and then the problem of constructing (tight) orthogonal arrays
can be treated in the broader context
of (tight) {\em designs} in association schemes~\cite{BBI}.
Informally speaking, design theory is devoted to finding subsets
that represent a good approximation of the whole space
such as, for example, a Hamming association scheme $H(n,q)$,
a Johnson association scheme $J(v,k)$,
or a real unit sphere $S^{d-1}$.
The design theories for these spaces
have been studied separately as {\em orthogonal arrays},
as {\em block designs}, and as {\em spherical designs}, respectively.

The lower bounds on $t$-designs were provided for $H(n,q)$ by Rao~\cite{R},
for $J(v,k)$ by  Ray-Chaudhuri and Wilson~\cite{RW} (see also~\cite{WR}),
and for $S^{d-1}$ by  Delsarte, Goethals and Seidel~\cite{DGS}.
We call a design {\em tight} if it achieves the corresponding lower bound.

Tight $t$-designs with large strength $t$
seem very rare in general~\cite{B, BD1979, BD1980, X},
while the classification problem of tight designs of small strength may lead
to fundamental problems in combinatorics: for example, any symmetric block design (and so a projective plane of order $k$)
is tight in the Johnson association scheme $J(v,k+1)$, and a Hadamard matrix of order $n+1$ gives rise to
a tight $2$-design in $H(n,2)$ (see Theorem \ref{thm:tightinH}).

Regarding tight designs of strength $4$,
Noda~\cite{N1979} showed the following in 1979.
\begin{theorem}\label{thm:coa}
Let $C$ be a tight $4$-design in a Hamming association scheme $H(n,q)$.
Then one of the following holds:
\begin{enumerate}
\item $(|C|,n,q)=(16,5,2)$,
\item $(|C|,n,q)=(243,11,3)$,
\item $(|C|,n,q)=(9a^2(9a^2-1)/2,(9a^2+1)/5,6)$,
where $a$ is a positive integer such that $a\equiv 0\pmod{3}$,
$a\equiv \pm1\pmod{5}$ and $a\equiv 5\pmod{16}$. \bqed
\end{enumerate}
\end{theorem}
The existence and uniqueness for (1) or (2) had been shown
(see Examples~\ref{H52} and~\ref{H113} in Section~\ref{sect:4class}).
It remains open to determine whether the third case exists or not.
In this paper, we show that there is no tight $4$-design as in Theorem~\ref{thm:coa}(3).

We briefly outline how to prove the non-existence result.
A unifying framework to study designs in the above mentioned settings
are {\em $Q$-polynomial association schemes}, which were introduced and developed by Delsarte~\cite{D}.
In particular, an important necessary condition
for the existence of tight designs in the Johnson association schemes
was established by Wilson (according to~\cite[Page 6]{D}, see also~\cite{RW}),
and the result was extended to tight designs
in $Q$-polynomial association schemes~\cite{D},
including the Johnson and Hamming association schemes,
and in the real unit sphere~\cite{DGS}.
Noda used this result (see Theorem~\ref{thm:wilson})
in his proof of Theorem~\ref{thm:coa}.

Furthermore, due to Delsarte's work, a tight $4$-design $C$ in $H(n,q)$
yields an association scheme of $2$ classes.
Decompose the vertex set $C$ into $q$ disjoint subsets which can be identified with orthogonal arrays of strength $3$ in $H(n-1,q)$.
We then apply an analogue of the result in~\cite{S}
to these subsets in the Hamming association schemes
to construct another association scheme $S$,
which, however, satisfies all known feasibility conditions.
The association scheme $S$ turns out to be $Q$-antipodal, and
this property allows us to calculate the {\em triple intersection numbers}
with respect to some triples of vertices of $S$.
Triple intersection numbers can be thought of as a generalization of intersection numbers
to triples of starting vertices instead of pairs, and, to our best knowledge,
their investigation has been previously used
to study strongly regular~\cite{CGS}
and distance-regular graphs~\cite{CJ,GK,JKT,JV2012,JV2017,U} only,
but not strictly $Q$-polynomial association schemes.
We hope that this approach will find more applications in the theory of association schemes.

In the case when $S$ corresponds to a tight $4$-design in $H(n,6)$,
as in Theorem~\ref{thm:coa}(3),
certain triple intersection numbers turn out to be non-integral,
which leads to a contradiction.
This completes the classification of tight $4$-designs in $H(n,q)$, or in other words,
that of tight orthogonal arrays of strength $4$.

The existence and classification problems of tight $2e$-designs in $H(n,q)$
have been ex\-ten\-sive\-ly studied.
Together with our result (see Corollary~\ref{cor:nonex}),
its current state is summarized in the following theorem.
\begin{theorem}\label{thm:tightinH}
The following hold.
\begin{enumerate}
\item {\rm{\cite[Theorem~7.5]{HSS}}} For $e=1,q=2$, a tight $2$-design in $H(n,2)$ is equivalent to a Hadamard matrix of order $n+1$.
\item {\rm{\cite[Theorem~3.1]{HSS}}} For $e=1,q\ge 3$, there exists a tight $2$-design in $H(q^2,q)$ for any prime power $q$.
\item {\rm{\cite{H}}} For $e\geq 3,q\geq 3$, there is no tight $2e$-design in $H(n,q)$.
\item {\rm{\cite{N1979}}} For $e=2$,
if there exists a tight $4$-design $C$ in $H(n,q)$,
then one of the following occurs:
\begin{enumerate}
\item $(|C|,n,q)=(16,5,2)$,
\item $(|C|,n,q)=(243,11,3)$.
\end{enumerate}
\item {\rm{\cite{MK}}} For $e=3,q=2$,
if there exists a tight $6$-design in $H(n,2)$, then $n=7,23$. \bqed
\end{enumerate}
\end{theorem}

The organization of the paper is as follows.
In Section~\ref{sect:Definitions}, we prepare basic notions for association schemes and orthogonal arrays.
In Section~\ref{sect:4class}, we show that tight $4$-designs in $H(n,q)$ yield $Q$-antipodal $Q$-polynomial association scheme of $4$ classes.
Finally, in Section~\ref{sect:Triple}, we analyze triple intersection numbers with respect to some triples of vertices of the scheme obtained in Section~\ref{sect:4class}
to conclude that there are no tight $4$-designs in $H(n,6)$.

\section{Preliminaries}\label{sect:Definitions}

In this section we prepare the notions needed in subsequent sections.

\subsection{Association schemes}\label{subsect:AS}
Let $X$ be a finite set of vertices
and $\{R_0,R_1,\ldots,R_D\}$ be a set of non-empty subsets of $X\times X$.
Let $A_i$ denote the adjacency matrix of the graph $(X,R_i)$
($0 \le i \le D$).
The pair $(X,\{R_i\}_{i=0}^D)$ is called
a {\em (symmetric) association scheme} of $D$ classes
if the following conditions hold:
\begin{enumerate}
\item $A_0 =I_{|X|}$, which is the identity matrix of size $|X|$,
\item $\sum_{i=0}^D A_i = J_{|X|}$,
which is the square all-one matrix of size $|X|$,
\item $A_i^\top=A_i$ ($1 \le i \le D$),
\item $A_iA_j=\sum_{k=0}^D p_{ij}^kA_k$,
where $p_{ij}^k$ are nonnegative integers  ($0 \le i,j \le D$).
\end{enumerate}
The nonnegative integers $p_{ij}^k$ are called {\em intersection numbers}.
The vector space $\mathcal{A}$ over $\mathbb{R}$ spanned by the matrices
$A_i$ forms an algebra.
Since $\mathcal{A}$ is commutative and semisimple,
there exists a unique basis of $\mathcal{A}$ consisting of
primitive idempotents $E_0=\frac{1}{|X|}J_{|X|},E_1,\ldots,E_D$.
Since the algebra $\mathcal{A}$ is closed
under the entry-wise multiplication denoted by $\circ$,
we define the {\em Krein parameters}
$q_{ij}^k$ ($0 \le i,j,k \le D$) by
$E_i\circ E_j=\frac{1}{|X|}\sum_{k=0}^D q_{ij}^kE_k$.
It is known that the Krein parameters are nonnegative real numbers
(see~\cite[Lemma~2.4]{D}).
Since both $\{A_0,A_1,\ldots,A_D\}$ and $\{E_0,E_1,\ldots,E_D\}$
form bases of $\mathcal{A}$,
there exists a matrix $Q=(Q_{ij})_{i,j=0}^D$ with
$E_i=\frac{1}{|X|}\sum_{j=0}^D Q_{ji}A_j$.
The matrix $Q$ is called the {\em second eigenmatrix}
of $(X,\{R_i\}_{i=0}^D)$.
An association scheme $(X,\{R_i\}_{i=0}^D)$
is said to be {\em $Q$-polynomial} if,
for some ordering of $E_1,\ldots,E_D$ and
for each $i$ ($0 \le i \le D$),
there exists a polynomial $v_i^*(x)$ of degree $i$
such that $Q_{ji}=v_i^*(Q_{j1})$ ($0 \le j \le D$).
It is also known that an association scheme is $Q$-polynomial
if and only if the matrix of  Krein parameters $L_1^*:=(q_{1j}^k)_{k,j=0}^D$
is a tridiagonal matrix with nonzero superdiagonal and subdiagonal~%
\cite[p.~193]{BI}
-- then $q_{ij}^k = 0$ holds whenever the triple $(i, j, k)$
does not satisfy the triangle inequality
(i.e., when $|i-j| < k$ or $i+j > k$).
For a $Q$-polynomial association scheme,
set $a_i^*=q_{1,i}^i$, $b_i^*=q_{1,i+1}^i$, and $c_i^*=q_{1,i-1}^i$.
These Krein parameters are usually gathered in the {\em Krein array}
$\{b_0^*, b_1^*, \dots, b_{D-1}^*; c_1^*, c_2^*, \dots, c_D^*\}$,
as the remaining Krein parameters of a $Q$-polynomial association scheme
can be computed from them.
We say that a $Q$-polynomial association scheme is {\em $Q$-antipodal}
if $b_i^*=c_{D-i}^*$ except possibly for $i = \lfloor D/2\rfloor$.
We simply say $Q$-antipodal association schemes for $Q$-antipodal $Q$-polynomial association schemes.
In a $Q$-antipodal association scheme,
we have $q_{ij}^k = 0$ whenever $i+j+k > 2D$
and the triple $(D-i, D-j, D-k)$ does not satisfy the triangle inequality.
See~\cite{DMM} and~\cite{MMW}
for more results on $Q$-antipodal association schemes.

There exists a matrix $G=(G_0\ G_1\ \cdots\ G_D)$
whose rows and columns are indexed by $X$,
satisfying that $GG^\top=|X|I_{|X|}$
and $G$ diagonalizes the adjacency matrices,
where $E_i=\frac{1}{|X|}G_iG_i^\top$ ($0 \le i \le D$)~\cite[p.~11]{D}.
We then define the $i$-th {\em characteristic matrix} $H_i$ of a non-empty subset $C$ of $X$ as the submatrix of $G_i$ that lies in the rows indexed by $C$. (Throughout this paper, a subset $C$ of $X$ is always non-empty.)

A subset $C$ of $X$
for a $Q$-polynomial association scheme $(X,\{R_i\}_{i=0}^D)$
is a {\em $t$-design} if its characteristic vector $\chi=\chi_C$
satisfies that $\chi^\top E_i\chi=0$ ($1 \le i \le t$).

\subsection{Triple intersection numbers}
For a triple of vertices $u, v, w \in X$ and integers $i$, $j$, $k$ ($0 \le i, j, k \le D$)
we denote by $\sV{u & v & w}{i & j & k}$
(or simply $[i\ j\ k]$ when it is clear
which triple $(u,v,w)$ we have in mind)
the number of vertices $x \in X$ such that
$(u, x) \in R_i$, $(v, x) \in R_j$ and $(w, x) \in R_k$.
We call these numbers {\em triple intersection numbers}.

Unlike the intersection numbers, the triple intersection numbers depend, in general, on the particular choice of $(u,v,w)$.
Nevertheless, for a fixed triple $(u,v,w)$,
we may write down a system of $3D^2$ linear Diophantine equations
with $D^3$ triple intersection numbers as variables,
thus relating them to the intersection numbers, cf.~\cite{JV2012}:
{\small
\begin{equation}
\sum_{\ell=0}^D [\ell\ j\ k] = p^U_{jk}, \qquad
\sum_{\ell=0}^D [i\ \ell\ k] = p^V_{ik}, \qquad
\sum_{\ell=0}^D [i\ j\ \ell] = p^W_{ij},
\label{eqn:triple}
\end{equation}
}
where $(v, w) \in R_U$, $(u, w) \in R_V$, $(u, v) \in R_W$, and
\[
[0\ j\ k] = \delta_{jW} \delta_{kV}, \qquad
[i\ 0\ k] = \delta_{iW} \delta_{kU}, \qquad
[i\ j\ 0] = \delta_{iV} \delta_{jU}.
\]
Moreover, the following theorem sometimes gives additional equations.

\begin{theorem}\label{thm:krein0}{\rm (\cite[Theorem~3]{CJ},
                                   cf.~\cite[Theorem~2.3.2]{BCN})}
Let $(X, \{R_i\}_{i=0}^D)$ be an association scheme of $D$ classes
with second eigenmatrix $Q$
and Krein parameters $q_{ij}^k$ $(0 \le i,j,k \le D)$.
Then,
\[
\pushQED{\bqed}
q_{ij}^k = 0 \quad \Longleftrightarrow \quad
\sum_{r,s,t=0}^D Q_{ri}Q_{sj}Q_{tk}\sV{u & v & w}{r & s & t} = 0
\quad \mbox{for all\ } u, v, w \in X. \qedhere
\popQED
\]
\end{theorem}

\subsection{Hamming association schemes and orthogonal arrays}\label{subsect:HammingOA}
Let $V=\{1,2,\ldots,q\}$ ($q\ge2$) and $X=V^n$.
For $x=(x_1,\ldots,x_n),y=(y_1,\ldots,y_n)\in X$, define the {\it Hamming distance} $d(x,y)$ to be the number of indices $i$ with $x_i\neq y_i$.
Suppose that $R_i=\{(x,y)\mid x,y\in X, d(x,y)=i\}$
for $i=0,1,\ldots,n$.
Then the pair $(X,\{R_i\}_{i=0}^n)$ is an association scheme,
which is called the {\em Hamming} association scheme $H(n,q)$.
The Hamming association scheme has the second eigenmatrix $Q=(K_{n,q,j}(i))_{i,j=0}^n$ and is a $Q$-polynomial
association scheme with the polynomials $v_i^*(x)=K_{n,q,i}(((q-1)n-x)/q)$,
where $K_{n,q,i}(x)$ is the Krawtchouk polynomial of degree $i$ defined as
$
K_{n,q,i}(x)=\sum_{j=0}^i(-1)^j (q-1)^{i-j}\binom{x}{j}\binom{n-x}{i-j}.
$

An {\it orthogonal array} $\OA(N,n,q,t)$
is an $N\times n$ matrix $M$ with entries the numbers $1,2,\ldots,q$
such that in any $N\times t$ submatrix of $M$
all possible row vectors of length $t$ occur equally often~\cite{HSS}.
Let $C$ be the set of row vectors of $M$.
We identify the orthogonal array $M$ with the subset $C$ in $X$.
It is known from~\cite[Theorem~4.4]{D}
that an orthogonal array $\OA(N,n,q,t)$
is equivalent to a $t$-design $C$ with $|C|=N$
in the Hamming association scheme $H(n,q)$.

For $t=2e$, the lower bound \eqref{eq:tight} on $N$ was given by Rao~\cite{R}.
An orthogonal array is said to be {\it complete} or {\it tight}
if it achieves equality in this bound.

The {\it degree set} of an orthogonal array $C$ is the set $S(C)$ of Hamming distances of $x,y$ among distinct $x,y\in C$,
and the {\it degree} $s$ of $C$ is defined as $s=|S(C)|$.
It is known that a tight $2e$-design has degree $e$~\cite[Theorem~5.12]{D}.
The following lemma characterizes designs in terms of their characteristic matrices.
The subsequent lemma and theorems are valid for any $Q$-polynomial association scheme, but we state these only for $H(n,q)$.
\begin{lemma}{\rm \cite[Theorem~3.15]{D}}\label{lem:cha}
Let $C$ be a subset in $H(n,q)$. The following conditions are equi\-valent:
\begin{enumerate}
\item $C$ is a $t$-design,
\item $H_k^\top H_\ell=\delta_{k\ell}|C|I \quad
\text{for} \quad 0\leq k+\ell\leq t$. \bqed
\end{enumerate}
\end{lemma}

Then the following theorems are crucial.
\begin{theorem}\label{thm:wilson}
Let $C$ be a tight $2e$-design in $H(n,q)$ with degree set $S=\{\alpha_1,\ldots,\alpha_e\}$.
Then $|C|\prod_{i=1}^e (1-x/\alpha_i)=\sum_{j=0}^e K_{n,q,j}(x)$ holds.
In particular, $\sum_{j=0}^e K_{n,q,j}(x)$
has exactly $e$ distinct integral zeros in the interval $[1,n]$.
\end{theorem}
\begin{proof}
See~\cite[Theorem~5.21]{D}.
\end{proof}
Let $C$ be a subset in $H(n,q)$ with degree set $S(C)=\{\alpha_1,\ldots,\alpha_s\}$.
Set $\alpha_0=0$.
Define $S_i=\{(x,y)\in C\times C\mid d(x,y)=\alpha_i\}$ ($0 \le i \le s$).
\begin{theorem}\label{thm:t2s-2}
Let $C$ be a $t$-design in $H(n,q)$ with degree $s$.
If $t\geq 2s-2$, then the pair $(C,\{S_i\}_{i=0}^s)$ is a $Q$-polynomial association scheme of $s$ classes.
\end{theorem}
\begin{proof}
Let $A_i$ be the adjacency matrix of the graph $(C,S_i)$ for each $i$, and
$\mathcal{A}$ the vector space spanned by $A_0,A_1,\ldots,A_s$.

Let $H_i$ be the $i$-th characteristic matrix of $C$ ($0 \le i \le s-1$).
Define $F_i=\frac{1}{|C|}H_i H_i^\top$ ($0 \le i \le s-1$).
Set $F_s=I-\sum_{j=0}^{s-1}F_j$.
Then $F_i=\frac{1}{|C|}\sum_{j=0}^s K_{n,q,i}(\alpha_j)A_j$ ($0 \le i \le s-1$)
by~\cite[Theorem~3.13]{D}
and $F_s = \frac{1}{|C|}\sum_{j=0}^s$ $f(\alpha_j)A_j$
where $f(z)=|C|\prod_{i=1}^s (1-z/\alpha_i)-\sum_{j=0}^{s-1} K_{n,q,j}(z)$.
Then $F_i\neq O$ and $F_i\in \mathcal{A}$ for each $i$.

By Lemma~\ref{lem:cha},
we have $F_i F_j=\delta_{ij}F_i$ ($0 \le i,j \le s-1$),
from which it follows that $F_iF_s=F_sF_i=O$ ($0 \le i \le s-1$)
and $F_s^2=F_s$.
These show that $\{F_0,F_1,\ldots,F_s\}$ form a set of mutually orthogonal idempotents of $\mathcal{A}$.
Therefore $\mathcal{A}$ is closed under matrix multiplication
and the pair $(C,\{S_i\}_{i=0}^s)$ is an association scheme.
Note that $F_i$ is written as a polynomial of degree $i$ in $F_1$
with respect to the entrywise product.
Therefore the scheme is $Q$-polynomial.
\end{proof}

\section{Tight $4$-designs in $H(n,q)$ and $Q$-antipodal association schemes of $4$ classes}\label{sect:4class}
Let $C$ be a tight $4$-design in $H(n,q)$ with degree set $S(C)=\{\alpha_1,\alpha_2\}$ where $\alpha_1,\alpha_2$ ($\alpha_1<\alpha_2$) are the zeros of $\sum_{j=0}^{2}K_{n,q,j}(x)=0$.
Set $\alpha_0=0$, and define $S_i=\{(x,y)\in C\times C\mid d(x,y)=\alpha_i\}$ for each $i$.
By Theorem~\ref{thm:t2s-2},  the pair $(C,\{S_i\}_{i=0}^2)$ is an association scheme of $2$ classes.
In this section, we decompose $S_1$ and $S_2$ into two subsets
so that a tight $4$-design in $H(n,q)$
yields a $Q$-anti\-podal association scheme of $4$ classes.

Define $C_i$ to be
\[
C_i=\{(x_2,\ldots,x_n)\mid (i,x_2,\ldots,x_n)\in C\}
\quad (1 \le i \le q).
\]
Then $C=\bigcup_{i=1}^q \{i\}\times C_i$ holds.
Note that $C_i$ is obtained from $C$ by deleting the first coordinate of the vectors with $x_1=i$ in $C$ and $|C_i|=|C|/q$ for each $i$.
Setting $\tilde{C}=\bigcup_{i=1}^{q}C_i$, we will consider further combinatorial structure on $\tilde{C}$ based on its partition $\tilde{C}=\bigcup_{i=1}^{q}C_i$.

Denote by $H_k^{(i)}$ the $k$-th characteristic matrix of $C_i$ in $H(n-1,q)$,
and observe that $C_i$ is a $3$-design with degree $2$ in $H(n-1,q)$.
First we claim the following lemma, which is crucial to construct an association scheme on $\tilde{C}$.

\begin{lemma}\label{lem:F}
Let $C$ be a tight $4$-design in $H(n,q)$.
Define $F_\ell^{(i,j)}$ to be
\begin{align*}
F_\ell^{(i,j)}=\frac{1}{\sqrt{|C_i||C_j|}}H_\ell^{(i)}(H_\ell^{(j)})^\top
\quad (1 \le i,j \le q, \ \ell\in\{0,1\})
\end{align*}
and
\begin{align*}
F_2^{(i,i)} = I-F_0^{(i,i)}-F_1^{(i,i)} \quad (1 \le i \le q).
\end{align*}
Then
$F_\ell^{(i,j)}F_{\ell'}^{(j,k)} = \delta_{\ell\ell'}F_\ell^{(i,k)}$
holds for $1 \le i,j,k \le q$ and $\ell,\ell'\in\{0,1\}$,
and $F_{2}^{(i,i)}F_{\ell}^{(i,j)}$ $= F_{\ell}^{(i,j)}F_{2}^{(j,j)}$ $= O$
holds for $1 \le i,j \le q$ and $\ell\in\{0,1\}$.
\end{lemma}
\begin{proof}
By Lemma~\ref{lem:cha}.
\end{proof}

Recall $\tilde{C}=\bigcup_{i=1}^{q}C_i$.
Then $\tilde{C}$ is a subset in $H(n-1,q)$ and $S(\tilde{C})=\{\alpha_1,\alpha_2,\alpha_1-1,\alpha_2-1\}$.
Define $\tilde{S}_0,\tilde{S}_1,\ldots,\tilde{S}_4$ by $\tilde{S}_0=\{(x,y)\in\tilde{C}\times \tilde{C}\mid d(x,y)=0\}$ and
\begin{align*}
\tilde{S}_{2i-1}&=\{(x,y)\in\tilde{C}\times \tilde{C}\mid d(x,y)=\alpha_i-1\},\\
\tilde{S}_{2i}&=\{(x,y)\in\tilde{C}\times \tilde{C}\mid d(x,y)=\alpha_i\}
\end{align*}
for $i\in\{1,2\}$.
The following theorem is the main theorem in this section.
\begin{theorem}\label{thm:qant4}
Let $C$ be a tight $4$-design in $H(n,q)$.
Then $(\tilde{C},\{\tilde{S}_i\}_{i=0}^4)$
is a $Q$-antipodal association scheme of $4$ classes
with Krein array
\[
\{(n-1)(q-1), (n-2)(q-1), 2(q-1), 1; 1, 2, (n-2)(q-1), (n-1)(q-1)\} .
\]
\end{theorem}
\begin{proof}
Let $A_i$ be the adjacency matrix of the graph $(\tilde{C},\tilde{S}_i)$
($0 \le i \le 4$),
and let $\mathcal{A}$ be the vector space
spanned by $A_0,A_1,\ldots,A_4$ over $\mathbb{R}$.

Since each $C_i$ is a $3$-design with degree $2$ in $H(n-1,q)$, $C_i$ provides an association scheme of $2$ classes by Theorem~\ref{thm:t2s-2}.
It follows from the proof of Theorem~\ref{thm:t2s-2} that the primitive idempotents of $C_i$ are $F_{0}^{(i,i)},F_{1}^{(i,i)},F_{2}^{(i,i)}:=I-F_{0}^{(i,i)}-F_{1}^{(i,i)}$, where $F_{\ell}^{(i,i)}=\frac{1}{|C_i|}H_\ell^{(i)}(H_\ell^{(i)})^\top$ for $\ell\in\{0,1\}$.

Now we define $E_0,E_1,\ldots,E_4$ as
\begin{align*}
E_i&=\frac{1}{q}\begin{pmatrix}
F_{i}^{(1,1)} & F_{i}^{(1,2)} & \cdots & F_{i}^{(1,q)} \\
F_{i}^{(2,1)} & F_{i}^{(2,2)} & \cdots & F_{i}^{(2,q)} \\
\vdots & \vdots & \ddots & \vdots \\
F_{i}^{(q,1)} & F_{i}^{(q,2)} & \cdots & F_{i}^{(q,q)}
\end{pmatrix} \text{ for } i\in\{0,1\},\\
E_2&=\phantom{\frac{q}{q}}
\begin{pmatrix}
F_{2}^{(1,1)} & O & \cdots & O \\
O & F_{2}^{(2,2)} & \cdots & O \\
\vdots & \vdots & \ddots & \vdots \\
O & O & \cdots & F_{2}^{(q,q)}
\end{pmatrix},\\
E_{4-i}&=\frac{1}{q}\begin{pmatrix}
(q-1)F_{i}^{(1,1)} & -F_{i}^{(1,2)} & \cdots & -F_{i}^{(1,q)} \\
-F_{i}^{(2,1)} & (q-1)F_{i}^{(2,2)} & \cdots & -F_{i}^{(2,q)} \\
\vdots & \vdots & \ddots & \vdots \\
-F_{i}^{(q,1)} & -F_{i}^{(q,2)} & \cdots & (q-1)F_{i}^{(q,q)}
\end{pmatrix} \text{ for } i\in\{0,1\}.
\end{align*}
Note that each $E_i$ is a non-zero matrix.
Since the matrices
\begin{align*}
\begin{pmatrix}
F_{i}^{(1,1)} & O & \cdots & O \\
O & F_{i}^{(2,2)} & \cdots & O \\
\vdots & \vdots & \ddots & \vdots \\
O & O & \cdots & F_{i}^{(q,q)}
\end{pmatrix},\quad
\begin{pmatrix}
O & F_{i}^{(1,2)} & \cdots & F_{i}^{(1,q)} \\
F_{i}^{(2,1)} & O & \cdots & F_{i}^{(2,q)} \\
\vdots & \vdots & \ddots & \vdots \\
F_{i}^{(q,1)} & F_{i}^{(q,2)} & \cdots & O
\end{pmatrix}
\end{align*}
are written as a linear combinations of $A_0,A_1,\ldots,A_4$, so are the matrices $E_0,E_1,\ldots,E_4$.
From Lemma~\ref{lem:F}, it follows that $E_0,E_1,\ldots,E_4$ are mutually orthogonal idempotents.
Thus $\mathcal{A}$ is closed under the matrix multiplication, and
$(\tilde{C},\{\tilde{S}_i\}_{i=0}^4)$ is an association scheme of $4$ classes with the primitive idempotents $E_0,E_1,\ldots,E_4$.
The second eigenmatrix $Q$ is given as:
\begin{align*}
Q=\begin{pmatrix}
 1 & (n-1)(q-1) & \frac{1}{2} (n^2\mm 3n\pp 2)(q\mm 1)^2 & (n-1)(q-1)^2 & q\mm 1 \\
 1 & \frac{1}{2}(q-2+d) & 0 & -\frac{1}{2}(q-2+d) & -1 \\
 1 & \frac{1}{2}(-q\mm 2\pp d) & \frac{1}{2}q(q-d)  & \frac{1}{2}(q\mm 1)(-q\mm 2\pp d) & q\mm 1 \\
 1 & \frac{1}{2}(q-2-d) & 0 & -\frac{1}{2}(q-2-d) & -1 \\
 1 & \frac{1}{2}(-q\mm 2\mm d) & \frac{1}{2}q(q+d) & \frac{1}{2}(q\mm 1)(-q\mm 2\mm d) & q\mm 1
\end{pmatrix},
\end{align*}
where $d=\sqrt{q^2 + 4(n-2)(q-1)}$.
Then the matrix $L_1^*$ is
\begin{align*}
L_1^*=\left(
\begin{array}{ccccc}
 0 & (n-1) (q-1) & 0 & 0 & 0 \\
 1 & q-2 & (n-2) (q-1) & 0 & 0 \\
 0 & 2 & n (q-1)-3 q+1 & 2 (q-1) & 0 \\
 0 & 0 & (n-2) (q-1) & q-2 & 1 \\
 0 & 0 & 0 & (n-1) (q-1) & 0 \\
\end{array}
\right).
\end{align*}
Therefore the scheme is a $Q$-antipodal scheme
with the given Krein array.
\end{proof}

\begin{remark}
The association scheme $(\tilde{C},\{\tilde{S}_i\}_{i=0}^4)$
is a fission scheme of $(C,$ $\{S_i\}_{i=0}^2)$ in the following way.
Let $\phi$ be a mapping from $C$ to $\tilde{C}$ defined by $\phi(x_1,x_2,\ldots,x_n)=(x_2,\ldots,x_n)$ and extended from $C\times C$ to $\tilde{C}\times \tilde{C}$ with respect to entrywise.
Then $\phi(S_0) = \tilde{S}_0$
and $\phi(S_i) = \tilde{S}_{2i-1}\cup \tilde{S}_{2i}$ for $i=1,2$ hold.
\end{remark}

\begin{example}\label{H52}
There exists a unique tight $4$-design in $H(5,2)$.
It is the dual code of the repetition code of length $5$.
By Theorem~\ref{thm:qant4},
it yields a $Q$-antipodal association scheme of $4$ classes
with Krein array $\{4, 3, 2, 1; 1, 2, 3, 4\}$
(i.e., the Hamming association scheme $H(4, 2)$).

\end{example}

\begin{example}\label{H113}
There exists a unique tight $4$-design in $H(11,3)$, namely the dual code of ternary Golay code.
By Theorem~\ref{thm:qant4},
it yields a $Q$-antipodal association scheme of $4$ classes
with Krein array $\{20, 18, 4, 1;$ $1, 2, 18, 20\}$.
\end{example}
\section{Triple intersection numbers
of a $Q$-antipodal association scheme of $4$ classes}\label{sect:Triple}
In this section we calculate triple intersection numbers of a $Q$-antipodal association scheme of $4$ classes obtained from a tight $4$-design in  $H((9a^2+1)/5,6)$ where $a$ is a positive integer such that $a\equiv 0\pmod{3}$, $a\equiv \pm1\pmod{5}$ and $a\equiv 5\pmod{16}$.

Let $C$ be a tight $4$-design in $H((9a^2+1)/5,6)$.
The corresponding association scheme $(\tilde{C},$ $\{\tilde{S}_i\}_{i=0}^4)$
has Krein array $\{9a^2-4, 9a^2-9, 10, 1; 1, 2, 9a^2-9, 9a^2-4\}$.
By substituting $3a = r$, we get the Krein array
$\{r^2-4, r^2-9, 10, 1; 1, 2, r^2-9, r^2-4\}$.
This parameter set is feasible for all odd $r \ge 5$
(i.e., the intersection numbers and multiplicities are nonnegative integers,
and the Krein parameters are nonnegative real numbers).

An association scheme with such parameters
has $r^2 (r^2 - 1)/2$ vertices and is $Q$-antipodal,
so many of its Krein parameters are zero.
For a chosen triple of vertices of the association scheme,
this allows us to augment the system of equations \eqref{eqn:triple}
with new equations derived from Theorem~\ref{thm:krein0}.
We used the {\tt sage-drg} package~\cite{Vdrg} (see also~\cite{V})
for the SageMath computer algebra system~\cite{Sage}
to derive the following result.

\begin{theorem} \label{thm:nonex}
Let $(X, \{R_i\}_{i=0}^4)$ be a $Q$-polynomial association scheme
with Krein array $\{r^2-4, r^2-9, 10, 1; 1, 2, r^2-9, r^2-4\}$.
Then $r = 9$.
\end{theorem}

\begin{proof}
Since the Krein array above
is obtained from the Krein array in Theorem~\ref{thm:qant4}
by setting $n = (r^2+1)/5$, $q = 6$,
we may write the corresponding second eigenmatrix as
\[
Q = \left(
\begin{array}{ccccc}
 1 & r^2-4 & \frac{1}{2} (r^2-4)(r^2-9) & 5 (r^2-4) & 5 \\
 1 & r+2 & 0 & -r-2 & -1 \\
 1 & r-4 & -6 (r-3) & 5 (r-4) & 5 \\
 1 & -r+2 & 0 & r-2 & -1 \\
 1 & -r-4 & 6 (r+3) & -5 (r+4) & 5 \\
\end{array}
\right) .
\]

As noted above, $r$ must be odd and at least $5$
for the intersection numbers $p_{ij}^k$ ($0 \le i, j, k \le 4$)
to be all nonnegative and integral.
In particular,
we have $p_{11}^1 = (r^2 - 3r + 6)(r^2 - 1)/12 > 0$ for all such $r$,
so we can choose $u, v, w \in X$ such that $(u, v), (u, w), (v, w) \in R_1$.

Solving the system of equations \eqref{eqn:triple} for the triple $(u, v, w)$
augmented by equations derived from Theorem~\ref{thm:krein0}
for each zero Krein parameter yields a one-parametrical solution
(see the notebook
\href{https://nbviewer.jupyter.org/github/jaanos/sage-drg/blob/master/jupyter/QPoly-d4-tight4design.ipynb}{\tt QPoly-d4-tight4design.ipynb}
on the {\tt sage-drg} package repository for computation details).
Let $\alpha = [1\ 2\ 3]$, and write $r = 2t+1$.
Then we may express
\[
[1\ 1\ 1] = t^4 + 2t^3 + 2t^2 - 3\alpha - {5r + 4 - 9/r \over 8} .
\]
Clearly, this expression can only be integral when $r$ divides $9$.
Since we must have $r \ge 5$,
this leaves $r = 9$ as the only feasible solution.
\end{proof}

\begin{corollary} \label{cor:nonex}
A tight $4$-design as in Theorem~\ref{thm:coa}(3) does not exist.
\end{corollary}

\begin{proof}
Let $(\tilde{C},\{\tilde{S}_i\}_{i=0}^4)$ be the association scheme
corresponding to a tight $4$-design in $H((9a^2+1)/5,6)$.
By Theorem~\ref{thm:qant4},
its Krein array matches that of Theorem~\ref{thm:nonex} with $r = 3a$,
from which $a = 3$ follows.
But this fails the condition $a \equiv \pm 1 \pmod{5}$,
so such a design cannot exist.
\end{proof}

Theorem~\ref{thm:nonex} allows for the existence
of a $Q$-polynomial association scheme
with Krein array $\{77, 72, 10, 1; 1, 2, 72, 77\}$.
No such scheme is known,
however such a scheme would have as a subscheme
a strongly regular graph
(i.e., an association scheme of $2$ classes)
with parameters $(v, k, \lambda, \mu) = (540, 154, 28, 50)$.
This parameter set is also feasible,
but no example is known, see~\cite{Bsrg}.

\Acknowledgements

\end{document}